\documentclass[12pt,reqno,a4paper,oneside]{amsart}

\usepackage{graphicx,url}
\usepackage{amssymb}
\usepackage{amsmath}
\usepackage{amsthm}
\usepackage{tikz}
\usetikzlibrary{patterns}
\usepackage{float}
\usepackage{chngcntr}
\usepackage{appendix}
\allowdisplaybreaks[4]

\newtheorem{theorem}{Theorem}
\newtheorem*{theorem*}{Theorem}
\newtheorem{lemma}[theorem]{Lemma}
\newtheorem{corollary}[theorem]{Corollary}
\newtheorem*{corollary*}{Corollary}

\newtheorem*{proposition*}{Proposition}

\theoremstyle{remark}
\newtheorem{remark}[theorem]{Remark}

\newcommand{\RR}{\mathbb{R}}

\newcommand{\NN}{\mathbb{N}}

\newcommand{\cH}{\mathcal{H}}

\newcommand{\dd}{\,\mathrm{d}}

\DeclareMathOperator{\vspan}{span}
\DeclareMathOperator{\supp}{supp}

\DeclareMathOperator{\spec}{spec}
\newcommand{\specess}{\mathop{\mathrm{spec}_\mathrm{ess}}}
\newcommand{\specd}{\mathop{\mathrm{spec}_\mathrm{disc}}}

\begin{document}

\title[$\delta'$-potentials on star graphs]{On Schr\"odinger operators with  $\delta'$-potentials supported on star graphs}

\author{Konstantin Pankrashkin}

\address{Carl von Ossietzky Universit\"at Oldenburg, Institut f\"ur Mathematik, 26111 Oldenburg, Germany}
\email{konstantin.pankrashkin@uol.de}

\author{Marco Vogel}
\address{Carl von Ossietzky Universit\"at Oldenburg, Institut f\"ur Mathematik, 26111 Oldenburg, Germany}
\email{marco.vogel@uol.de}

\begin{abstract}
The spectral properties of two-dimensional Schr\"odinger operators with $\delta'$-potentials supported on star graphs are discussed. We describe the essential spectrum and give a complete description of situations in which the discrete spectrum is non-trivial but finite. A more detailed study is presented for the case of a star graph with two branches, in particular, the small angle asymptotics for the eigenvalues is obtained.

\bigskip

\noindent The final version is published in J. Phys. A: Math. Theor. 55 (2022) 295201 and is available at \url{https://doi.org/10.1088/1751-8121/ac775a}.
\end{abstract}

\maketitle

\section{Introduction}

The $\delta'$-potentials represent an idealized model of quantum-mechanical interactions supported by small sets, and they build an important class of so-called solvable models in quantum mechanics, as many explicit formulas for various spectral characteristics of the respective Schr\"odinger operators can be obtained \cite{AGHH}. They are also of interest from the point of view of the spectral geometry, as one deals with various links between the geometric properties of the potential support and the eigenvalues of the associated differential operators, see e.g. \cite{leaky,lotor2}. In the present paper we prove some results on the spectral analysis for Schr\"odinger operators with $\delta'$-potentials supported by star graphs in two dimensions (in this precise case, the star graph geometry is completely determined by the angles between the branches). While some statements on such operators can be found in the literature, we are not aware of any systematic study of this class
and believe that the paper can have a consolidating effect.  It seems that singular potentials supported by star graphs were first considered in \cite{en} for $\delta$-potentials as a model
of quantum branching structures. Both $\delta$ and $\delta'$ potentials model interactions which are strongly concentrated near $\Gamma$ but with different limiting behavior in the normal direction \cite{enz}, so these two classes require separate study. From the methodological point of view, our main observation is that a rather detailed spectral picture for the $\delta'$-case can still be deduced from a comparison with Robin Laplacians and $\delta$-potentials studied by various authors in earlier works.

\begin{figure}
	\centering

\begin{tabular}{ccc}
	\begin{tikzpicture}
		\draw [->,color=gray, line width=0.5mm] (-1,0) -- (3,0);
		\draw [->,color=gray, line width=0.5mm] (0,-1.5) -- (0,2.5);
		\draw [domain=0:2.0, line width=3.0pt] plot (\x,\x);
		\draw [domain=-1:0, line width=3.0pt] plot (\x,-2*\x);
		\draw [domain=-0.7:0, line width=3.0pt] plot (\x,3*\x);
		\draw (0.8,0) arc (0:45:0.8);
		\draw (1.4,0) arc (0:115:1.4);
		\draw (2,0) arc (0:252:2);
		\coordinate (a) at (1,0);
		\coordinate (b) at (0.8,1.2);
		\coordinate (c) at (0.8,1.8);
		\coordinate (d) at (2.3,2);
		\coordinate (e) at (-1,2);
		\coordinate (f) at (0.1,-2);
		\draw (a) node[above] {$\theta_1$};
		\draw (b) node[above] {$\theta_2$};
		\draw (c) node[above] {$\theta_3$};    
		\draw (d) node[below] {$\Gamma_1$};
		\draw (e) node[above] {$\Gamma_2$};
		\draw (f) node[left] {$\Gamma_3$};		
		
	\end{tikzpicture} &\qquad
&\begin{tikzpicture}
	\draw [domain=0:2.0, line width=3.0pt] plot (\x,\x);
	\draw [domain=-1:0, line width=3.0pt] plot (\x,-2*\x);
	\draw [domain=-0.7:0, line width=3.0pt] plot (\x,3*\x);
	\draw (0.5,0.5) arc (45:-110:0.7);
	\draw (0.75,0.75) arc (45:115:1.05);
	\draw (-0.63,1.35) arc (115:250:1.5);
	\coordinate (a) at (1,-1);
	\coordinate (b) at (0.3,1.0);
	\coordinate (c) at (-0.5,0);
	\coordinate (d) at (2.3,2);
	\coordinate (e) at (-1,2);
	\coordinate (f) at (0.1,-2);
	\coordinate (v3) at (1,0);
	\coordinate (v2) at (-1.5,0);
	\coordinate (v1) at (0.6,1.6);
	\draw (a) node[above] {$2\varphi_3$};
	\draw (b) node[above] {$2\varphi_1$};
	\draw (c) node[left] {$2\varphi_2$};    
	\draw (d) node[below] {$\Gamma_1$};
	\draw (e) node[above] {$\Gamma_2$};
	\draw (f) node[left] {$\Gamma_3$};		

	\draw (v3) node[right] {\Large $V_3$};		
	\draw (v1) node[above] {\Large $V_1$};		
	\draw (v2) node[left] {\Large $V_2$};		
	
\end{tikzpicture}\\
(a) && (b)
\end{tabular}

	\caption{(a) Star graph $\Gamma$ with three branches $\Gamma_j$. (b) The associated decomposition into infinite sectors $V_j$.}\label{fig0}
	
\end{figure}
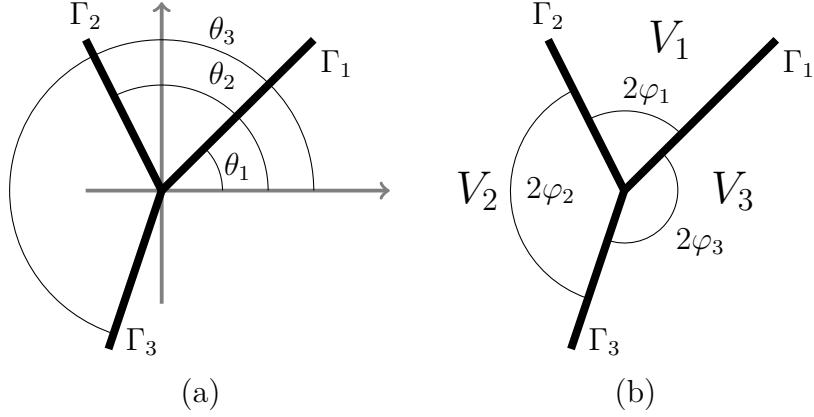

Let us  introduce the main definitions. Let $(r,\theta)$ be the standard polar coordinates in $\RR^2$, then
the set
\[
\Gamma:=\bigcup_{j=1}^{M}\Gamma_{j} \text{ with } \Gamma_{j}=\{(r,\theta_{j}):r\geq0\},
\]
with $M\in\NN$ and $0\leq\theta_{1}<...<\theta_{M}<2\pi$, is called a star graph with $M$ branches, see Fig.~\ref{fig0}(a). For $\alpha\in\RR$, the set $\Gamma$
gives rise to a two-dimensional ``geometric'' Schr\"odinger operator $T_{\Gamma,\alpha}$, which is defined as the unique self-adjoint operator in $L^2(\RR^2)$
generated by the closed, lower semibounded, symmetric sesquilinear form $t_{\Gamma,\alpha}$ given by
\[
 t_{\Gamma,\alpha}(u,u)=\int_{\mathbb{R}^{2}\setminus \Gamma}|\nabla u|^{2}\dd x+\alpha \int_\Gamma \big|[u]\big|^2\dd \sigma
\]		
with domain $D(t_{\Gamma,\alpha})=H^1(\RR^2\setminus\Gamma)$, where $\dd\sigma$ stands for the integration with respect to arclength and $[u]$ denotes the jump of $u$ at $\Gamma$, i.e. the difference of the values of $u$ at the two sides of $\Gamma$. Namely, let $N$ be a unit normal field at $\Gamma$, smooth outside the origin. Then for $s\in \Gamma_j\setminus\{0\}$ and $u$ smooth at the both sides of $\Gamma_j$ one has
\[
[u](s):=\lim_{t\to 0^+}u\big (s+tN(s)\big)-u\big(s-tN(s)\big),
\]
 which uniquely extends by density to a bounded map $H^1(\RR^2\setminus\Gamma)\ni u\mapsto [u]\in L^2(\Gamma)$ using the Sobolev trace theorem. One often uses the formal expression $T_{\Gamma,\alpha}=-\Delta+\alpha\delta'_{\Gamma,\alpha}$, and it can be shown that $T_{\Gamma,\alpha}$ acts as  $u\mapsto -\Delta u$ (in the sense of distributions on $\RR^2\setminus\Gamma$) on the functions $u$ satisfying special transmission conditions on $\Gamma$,
\[
\partial^+_N u=\partial^-_N u=:\partial_N u, \quad \alpha[u]=\partial_N u, \quad
\partial^\pm_N u(s):=\lim_{t\to0^+} N(s)\cdot \nabla u\big(s\pm tN(s)\big),
\]
where all values are understood in the sense of Sobolev traces: we refer to \cite{bel14,bll,er} for a detailed discussion
and a rigorous version of the approach through the transmission conditions.

The case of $\alpha\ge 0$ is relatively simple. On one hand,
one easily observes that $T_{\Gamma,\alpha}$ is a non-negative operator, hence its spectrum is contained in $[0,\infty)$.
Consider the $M$ connected components of $\RR^2\setminus\Gamma$, i.e. the infinite sectors
bounded by $\Gamma_j$ and $\Gamma_{j+1}$ (assuming the cyclic enumeration convention),
\begin{equation}
	\label{eqvj}
	V_{j}:=\big\{(r,\theta):\theta\in(\theta_{j},\theta_{j+1}),\ r>0\big\},\quad \theta_{M+1}:=2\pi+\theta_{1}.
\end{equation}
If one denotes 
\begin{equation}
	\label{phij}
	\varphi_j:=\dfrac{\theta_{j+1}-\theta_j}{2},
\end{equation}
then the opening angle of $V_j$ is $2\varphi_j$, see Fig.~\ref{fig0}(b). Remark that all $V_j$ are quasiconical domains, i.e. contain arbitrary large balls \cite[Sec.~49]{glaz}, and that $T_{\Gamma,\alpha}u$ coincides with $-\Delta u$ if $u$ is supported in one of the $V_j$, then it follows that $[0,\infty)\subset\spec T_{\Gamma,\alpha}$ (it seems that the argument
is not sufficiently known: we prefer to give an elaborate proof in Appendix \ref{specapp}). It follows that
$\spec T_{\Gamma,\alpha}=[0,\infty)$ for all $\alpha\ge 0$, in particular, no discrete spectrum is present. Hence, from now on we are interested in the case of an ``attractive interaction''
\[
\alpha<0.
\]
We prove the following results:
\begin{theorem}\label{thm1} Let $\alpha<0$, then for any choice of $M$ and $\theta_j$ one has:
\[
\specess T_{\Gamma,\alpha}=[-4\alpha^2,\infty),
\quad 
\specd T_{\Gamma,\alpha} \text{is at most finite.}
\]
Furthermore, the discrete spectrum of $T_{\Gamma,\alpha}$ is
\begin{enumerate}
 \item empty for
 \begin{itemize}
 	\item  $M=1$, i.e. if $\Gamma$ is a half-line,
 	 \item $M=2$ with $\theta_2-\theta_1=\pi$, i.e. if $\Gamma$ is a line,
 \end{itemize}
\item non-empty in all other cases.
\end{enumerate}
%
\end{theorem}	
Remark that a standard dilation argument shows the unitary equivalence
\begin{equation}
	 \label{unit1}
	T_{\Gamma,\alpha}\simeq \alpha^2 T_{\Gamma}, \qquad T_\Gamma:=T_{\Gamma,-1},
\end{equation}
so, in fact, it is sufficient to consider the case $\alpha=-1$ only. Theorem \ref{thm1} is proved in Sections~\ref{secess} and~\ref{secdisc}, and further observations will be presented through the whole text. In particular, in Corollary \ref{corolast} we show that the discrete spectrum of $T_{\Gamma,\alpha}$ can be arbitrarily large.

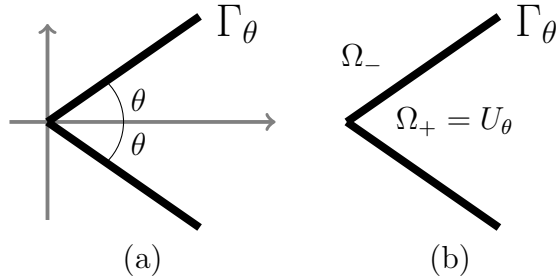
\begin{figure}[b]
	\centering

\begin{tabular}{ccc}
	\begin{tikzpicture}
		\coordinate (a) at (1.2,0);
		\coordinate (b) at (-0.5,0.5);
		\coordinate (c) at (2,0.2);
		\coordinate (d) at (2.5,1.7);
		\draw [->,color=gray, line width=0.5mm] (-0.5,0) -- (3,0);
		\draw [->,color=gray, line width=0.5mm] (0,-1.3) -- (0,1.3);
		\draw [domain=0:2, line width=3.0pt] plot (\x,-0.7*\x);
		\draw [domain=0:2, line width=3.0pt] plot (\x,0.7*\x);
		\draw (1,0) arc (0:45:0.75);
		\draw (1,0) arc (0:-45:0.75);
		\draw (a) node[above] {$\theta$};
		\draw (a) node[below] {$\theta$};
		\draw (d) node[below] {\Large $\Gamma_{\theta}$};
	\end{tikzpicture}
	& \qquad &
	\begin{tikzpicture}
	\coordinate (a) at (1.2,0);
	\coordinate (b) at (0.2,0.5);
	\coordinate (c) at (0.5,0);
	\coordinate (d) at (2.5,1.7);
	\draw [domain=0:2, line width=3.0pt] plot (\x,-0.7*\x);
	\draw [domain=0:2, line width=3.0pt] plot (\x,0.7*\x);
	\draw (b) node[above] {$\Omega_{-}$};    
	\draw (c) node[right] {$\Omega_{+}=U_\theta$};
	\draw (d) node[below] {\Large $\Gamma_{\theta}$};
\end{tikzpicture}\\
(a) && (b)
\end{tabular}	
	\caption{(a) Visualization of $\Gamma_{\theta}$. (b) The associated sectors $\Omega_+=U_\theta$ and $\Omega_-$.}\label{fig1}
\end{figure}

We then discuss in greater detail the case $M=2$, which allows for a more explicit analysis.
For $\theta\in(0,\frac{\pi}{2})$, let $\Gamma_{\theta}$ denote the star graph with two branches forming an angle of $2\theta$, 
more precisely
\begin{equation}
	\label{eqgt}
\Gamma_{\theta}=\big\{(x,y)\in\mathbb{R}^{2}:x\ge0\text{ and }\lvert y\rvert= x \tan \theta\big\},
\end{equation}
where $(x,y)$ are Cartesian coordinates, see Fig.~\ref{fig1}(a),
and consider the associated operators
\begin{equation}
	\label{eqhtt}
H_{\theta,\alpha}:=T_{\Gamma_\theta,\alpha},
\quad
H_{\theta}:=H_{\theta,-1}.
\end{equation}
The unitary equivalence \eqref{unit1} takes the form
\begin{equation}
	\label{unit2}
	H_{\theta,\alpha}\simeq \alpha^2 H_\theta.
\end{equation}
The sesquilinear forms for $H_{\theta,\alpha}$ and $H_\theta$ will be denoted respectively by $h_{\theta,\alpha}$
and $h_\theta$, in particular,
\[
h_{\theta,\alpha}(u,u)=\int_{\RR^2\setminus\Gamma_\theta}|\nabla u|^2\dd x+\alpha\int_{\Gamma_\theta}|[u]|^2\dd\sigma,
\quad
D(h_{\theta,\alpha})=H^1(\RR^2\setminus\Gamma_\theta).
\]

Then more detailed statements are possible:
\begin{theorem}\label{thm2}
	For any $\alpha<0$ and any $\theta\in(0,\frac{\pi}{2})$ there holds
	\[
	\specess H_{\theta,\alpha}=[-4\alpha^2,\infty), \quad \specd H_{\theta,\alpha} \text{ is non-empty and finite.}
	\]

Let $\alpha<0$ be fixed. For $\theta\to0^+$ the number of discrete eigenvalues of $H_{\theta,\alpha}$
grows to infinity, and the individual eigenvalues are continuous and strictly increasing in $\theta$, namely: 
\begin{enumerate}
	\item if for some $n\in\NN$ and $\theta_n\in(0,\frac{\pi}{2})$ the operator $H_{\theta_{n},\alpha}$ has at least $n$ discrete eigenvalues, then $H_{\theta,\alpha}$ has at least $n$ discrete eigenvalues for any $\theta\in(0,\theta_n)$, and the function $(0,\theta_n)\ni\theta\mapsto E_n(H_{\theta,\alpha})$ is continuous and strictly increasing.
	\item for any $n\in\NN$ there exists $\theta_n\in(0,\frac{\pi}{2})$ such that $H_{\theta,\alpha}$ has at least $n$ discrete eigenvalues for any $\theta\in(0,\theta_n)$, and
	\[
	E_{n}(H_{\theta,\alpha})=-\frac{\alpha^2}{(2n-1)^{2}\theta^{2}}+O\Big(\frac{1}{\theta}\Big) \text{ as } \theta\to 0^+.
	\]
\end{enumerate}
\end{theorem}	
In fact, the proofs of the two main theorems are closely interrelated, as will be seen below. It is clear that all statements of Theorem \ref{thm2} concerning the essential spectrum and the cardinality of the discrete spectrum follow directly from Theorem \ref{thm1} and are only included for the sake of completeness. The remaining statements dealing with the dependence
of eigenvalues on the angle $\theta$ are proved in Section \ref{secangle}.

\begin{remark}
While the sets $\Gamma_\theta$ and the operators $H_{\theta,\alpha}$ can be defined for any $\theta\in(0,\pi)$, considering only $\theta\in(0,\frac{\pi}{2})$ is a normalization and not a restriction.
In fact, for $\theta=\frac{\pi}{2}$ the set $\Gamma_\theta$ is a line, which is covered
by Theorem \ref{thm1} (no discrete spectrum), and for $\theta\in(\frac{\pi}{2},\pi)$ the operator $H_{\theta,\alpha}$ is unitarily equivalent to $H_{\theta',\alpha}$ with $\theta':=\pi-\theta\in(0,\frac{\pi}{2})$
as $\Gamma_{\theta}$ is an isometric copy of $\Gamma_{\theta'}$ obtained by the reflection on the $y$-axis. In particular,
$E_n(H_{\theta,\alpha})=E_n(H_{\pi-\theta,\alpha})$ for any $n$.
\end{remark}

\section{Preliminaries}

In this section we collect our main tools and recall the definition of some important model operators. We will denote
\begin{equation*}
	\mathbb{R}_{+}^{2}:=\mathbb{R}\times(0,\infty)\text{ and } \mathbb{R}_{-}^{2}:=\mathbb{R}\times(-\infty,0).
\end{equation*}
Our main results will be obtained by comparing $\delta'$-potentials with other operators associated with similar geometries and studied in previous works. The approach is inspired, in particular, by \cite{bll,lotor}.

\subsection{Min-Max Principle}\label{sec-min-max}
We will recall some notation and basic facts on the min-max principle, see e.g. \cite[Section XIII.1]{RS4}. All these facts are well known, but we prefer to state them here in a suitable form as they are intensively used in the subsequent text.

 Let $\mathcal{H}$ be an infinite-dimensional Hilbert space, $t$ be a closed, lower semibounded, densely defined, symmetric sesquilinear form with domain $D(t)$, and $T$ be the unique self-adjoint operator $T$ in $\cH$ generated by $t$.
 We denote by $\spec T$, $\specess T$, $\specd T$ respectively the spectrum, the essential spectrum and the discrete spectrum of $T$.
  For $n\in\mathbb{N}$, we denote the $n$-th discrete eigenvalue of $T$ (if it exists), when enumerated in the non-decreasing order and taking the multiplicities into account, by $E_{n}(T)$. We also put $\Sigma:=\inf\specess T$ (with $\Sigma:=+\infty$ for $\inf\specess T=\emptyset$).
We also consider the non-decreasing sequence of numbers
\begin{equation*}
\Lambda_{n}(T):=\inf_{\substack{V\subset D(t)\\ \dim V=n}}\sup_{u\in V\setminus\{0\}}\frac{t(u,u)}{\langle u,u\rangle}, \quad n\in\NN.
\end{equation*}
It is known (min-max principle) that only two cases are possible:
\begin{enumerate}
\item For all $n\in\mathbb{N}$ there holds $\Lambda_{n}(T)<\Sigma$. Then $E_n(T)=\Lambda_{n}(T)$ for any $n\in\NN$, i.e. the spectrum of $T$ is purely discrete, and $\Lambda_n(T)$ is its $n$-th eigenvalue for any $n$.

\item For some $N\in\NN\cup\{0\}$ there holds $\Lambda_{N+1}(T)\geq\Sigma$, while $\Lambda_n(T)<\Sigma$ for all $n\le N$.
Then $T$ has exactly $N$ discrete eigenvalues in $(-\infty,\Sigma)$, and $E_n(T)=\Lambda_{n}(T)$ for $n\in\{1,\dots,N\}$, while $\Lambda_{n}(T)=\Sigma$ for all $n\geq N+1$.
\end{enumerate}
In particular, in all cases there holds $\lim_{n\to\infty}\Lambda_n(T)=\Sigma$, and
if for some $n\in\NN$ one has $\Lambda_{n}(T)<\Sigma$, then $E_n(T)=\Lambda_n(T)$.

The above characterization of eigenvalues will mostly be used in combination with the following
construction. Let $T$ and $T'$ be self-adjoint operators in Hilbert spaces $\mathcal{H}$ and $\mathcal{H}'$ generated by closed, lower semibounded, densely defined, symmetric sesquilinear forms $t$ and $t'$.
Assume that there exists a linear isometry $J:\cH'\rightarrow \cH$ with
\[
J D(t')\subset D(t), \quad t(Ju,Ju)\leq t'(u,u) \text{ for all } u\in D(t'),
\]
which will be termed in the subsequent text as
\[
T\le T' \text{ using $J$.}
\]
It follows directly from the above definition that for any $n\in\mathbb{N}$ there holds $\Lambda_{n}(T)\leq\Lambda_{n}(T')$, and the min-max principle implies the following relations:
\begin{enumerate}
	\item $\inf\specess T\le \inf\specess T'$,
	\item if for some $a\in \RR$ the spectrum of $T$ in $(-\infty,a)$ is purely discrete and consists of $n$ eigenvalues,
	then the spectrum of $T'$ in $(-\infty,a)$ is purely discrete and consists of \emph{at most} $n$ eigenvalues,
	\item if for some $a\in\RR$ the spectrum of $T'$ in $(-\infty,a)$ is purely discrete and consists of $n'$ eigenvalues,
	then either (i) $\inf\specess T<a$ or (ii) the spectrum of $T$ in $(-\infty,a)$ is purely discrete and consists of \emph{at  least} $n'$ eigenvalues.
\end{enumerate}

\subsection{Comparing eigenvalues of star graphs}

The min-max principle has the following direct application to star graphs:
\begin{lemma}\label{lem0}
Let $\Gamma$ and $\Tilde \Gamma$ be star graphs with $\Gamma\subset\Tilde\Gamma$, then
for any $\alpha<0$ one has $T_{\Gamma,\alpha}\ge T_{\Tilde\Gamma,\alpha}$ using the identity map.
\end{lemma}

\begin{proof}
Let $\Gamma_j$, $j=1,\dots M$, be the branches of $\Tilde \Gamma$, then
$\Gamma=\bigcup_{j\in K} \Gamma_j$ for some $K\subset\{1,\dots,M\}$.
We clearly have
\[
D(t_{\Gamma,\alpha})=H^1(\RR^2\setminus\Gamma)\subset H^1(\RR^2\setminus\Tilde\Gamma)=D(t_{\Tilde\Gamma,\alpha}),
\]
and for any $u\in D(t_{\Gamma,\alpha})$  there holds, due to $\alpha<0$ and $|[u]|^2\ge 0$,
\begin{align*}
t_{\Gamma,\alpha}(u,u)&=\int_{\RR^2\setminus \Gamma} |\nabla u|^2\dd x+\alpha\sum_{j\in K} \int_{\Gamma_j} \big|[u]\big|^2\dd\sigma\\
&\ge \int_{\RR^2\setminus \Tilde \Gamma} |\nabla u|^2\dd x+\alpha\sum_{j=1}^M \int_{\Gamma_j} \big|[u]\big|^2\dd\sigma=t_{\Tilde\Gamma,\alpha}(u,u),
\end{align*}
which is the sought inequality.
\end{proof}

\subsection{One dimensional $\delta'$-potential}\label{delta'-section}
Consider the following  closed, lower semibounded, densely defined, symmetric sesquilinear form $b$ in $L^2(\RR)$,
\begin{equation}\label{one-d-delta'}
b(f,f)=\int_{\mathbb{R}}\lvert f'\rvert^{2}\dd x-\lvert f(0+)-f(0-)\rvert^{2},\quad D(b)=H^{1}(\mathbb{R}\setminus\{0\}),
\end{equation}
and let $B$ be the self-adjoint operator in $L^2(\RR)$ associated with $b$. One easily shows that
$B$ acts as $f\mapsto -f''$ (in the sense of distributions on $\RR\setminus\{0\}$)  with domain
\[
D(B)=\big\{ f\in H^{2}(\mathbb{R}\setminus\{0\}):\ f'(0+)=f'(0-),\ f(0+)-f(0-)=-f'(0+)\big\},
\]
which is the Schr\"odinger operator with an attractive $\delta'$-potential of intensity $1$ at the origin, see \cite{AGHH}.
An easy direct computation shows that the only discrete eigenvalue of $B$ is $E_{1}(B)=-4$ with the corresponding eigenfunction
\begin{equation*}
\psi_{1}(x)=\mathrm{sgn}(x)e^{-2\lvert x\rvert}=
\begin{cases}
-e^{2x}, &x<0,
\\
e^{-2x}, &x>0.
\end{cases}
\end{equation*}

\subsection{Robin Laplacian in a sector}\label{sec-robin}

For $\theta\in(0,\pi)$ consider the infinite sector
\begin{equation}\label{equt}
U_\theta:=\big\{(x,y): |\arg(x+iy)|<\theta\big\},
\end{equation}
see Fig.~\ref{fig1}(b). For $\gamma>0$ consider the closed, lower semibounded, densely defined, symmetric sesquilinear form
\begin{equation}\label{magdaQF2}
	q_{\theta}^{\gamma}(v,v)=\int_{U_\theta}|\nabla v|^2\dd x \dd y-\gamma\int_{\partial U_\theta}|v|^2\dd\sigma,
	\quad v\in H^{1}(U_\theta),
\end{equation}
and denote by $Q^\gamma_\theta$ the self-adjoint operator in $L^2(U_\theta)$ generated by $q^\gamma_\theta$.
Remark again the unitary equivalence
\begin{equation}
	 \label{unit3}
Q^\gamma_\theta\simeq \gamma^2 Q_\theta, \quad Q_\theta:=Q^1_\theta.
\end{equation}
The following results were obtained in \cite{kp18,lp}:
\begin{enumerate}
	\item For any $\gamma>0$ and $\theta\in(0,\pi)$ one has
	\[
	\specess Q^\gamma_\theta=[-\gamma^2,\infty),
	\quad
	\specd Q^\gamma_\theta \text{ is at most finite,}
	\]
	\item Let $\gamma>0$, then $\specd Q^\gamma_\theta$ is empty for $\theta\ge \frac{\pi}{2}$
	and non-empty for $\theta<\frac{\pi}{2}$,
	\item Let $\gamma>0$. For any $n\in\NN$ there exists $\theta_n>0$ such that $Q^\gamma_\theta$ has $n$ discrete eigenvalues in $(-\infty,-\gamma^2)$ for any $\theta\in(0,\theta_n)$, and
	\[
	E_n(Q^\gamma_\theta)=-\frac{\gamma^{2}}{(2n-1)^{2}\theta^{2}}+O(1)\text{ as }\theta\to 0^+.
	\]
	Due to the unitary equivalence \eqref{unit3} the numbers $\theta_n$ can be chosen independent of $\gamma>0$.
\end{enumerate}
A family of eigenvalues and eigenfunctions of $Q^\gamma_\theta$ was recently obtained in an explicit form~\cite{lyal}, but it remains open if this family exhausts the whole discrete spectrum.

\subsection{Schr\"odinger operator with a $\delta$-potential on a broken line}\label{sec-delta}

For $\theta\in(0,\frac{\pi}{2})$ consider again the broken line $\Gamma_\theta$ defined in \eqref{eqgt}.
For $\gamma>0$ consider the following closed, lower semibounded, symmetric sesquilinear form $a^\gamma_\theta$ in $L^2(\RR^2)$,
\[
a^\gamma_\theta(u,u)=\int_{\RR^2} |\nabla u|^2\dd x-\gamma\int_{\Gamma_\theta}|u|^2\dd\sigma,
\quad D(a^\gamma_\theta)=H^1(\RR^2),
\]
and denote by $A^\gamma_\theta$ the associated self-adjoint operator in $L^2(\RR^2)$, then it is known from~\cite{ei} that for any $\gamma>0$ and any $\theta\in (0,\frac{\pi}{2})$ there holds
\[
\specess A^\gamma_\theta=\big[ -\tfrac{1}{4}\gamma^2,\infty\big),
\quad
\specd A^\gamma_\theta \text{ is non-empty.}
\]
Remark that an easy proof of the non-emptyness of the discrete spectrum was given in \cite{p17}.

\section{The essential spectrum}\label{secess}

In this section we prove all statements of Theorem \ref{thm1} concerning the essential spectrum.
In view of the unitary equivalence \eqref{unit1} it is sufficient to prove them for the case $\alpha=-1$, i.e. for the operator $T_\Gamma$ only.

\begin{lemma}\label{lem1}
	For any star graph $\Gamma$ there holds $[-4,\infty)\subset \specess T_\Gamma$.
\end{lemma}

\begin{proof}
The proof is quite standard and uses Weyl sequences.
	As the rotation of $\Gamma$ around the origin produces a unitary equivalent operator (with the same essential spectrum), without loss of generality we assume $\theta_1=0$, i.e. that the first branch of $\Gamma$ coincides with the positive $x$-half-axis.

	Pick a function $\Phi\in C^{\infty}(\mathbb{R})$ with $0\le\Phi\le 1$ such that
$\Phi(t)=0$  for $t\leq0$, $\Phi(t)=1$ for $t\geq1$,
and for a suitable (to be defined later) $a>0$ consider the functions
\[
\chi_{n}:\  x\mapsto \Phi(2n-x)\Phi(x-n),\quad
\widetilde{\chi}_{n}:\ y\mapsto \Phi(an-y)\Phi(y+an).
\]
Then it is clear that $\chi_{n},\widetilde{\chi}_{n}\in C_{c}^{\infty}(\mathbb{R})$ with
$0\leq\chi_{n},\widetilde{\chi}_{n}\leq1$, and
		\begin{align*}
			\chi_{n}(x)&=0\text{ if }x\in(-\infty,n)\cup(2n,\infty),
			\\
			\chi_{n}(x)&=1\text{ if }x\in(n+1,2n-1),
			\\
			\widetilde{\chi}_{n}(y)&=0\text{ if }y\in(-\infty,-an)\cup(an,\infty),
			\\
			\widetilde{\chi}_{n}(y)&=1\text{ if }y\in(-an+1,an-1).
		\end{align*}
Let $k\in\RR$ and define, for large $n\in\NN$,
\begin{equation}\label{weylseq}
	f_{n}:(x,y)\mapsto e^{ikx}\psi_1(y)\chi_{n}(x)\widetilde{\chi}_{n}(y),
\end{equation}
where $\psi_1$ is the first eigenfunction of the one-dimensional $\delta'$-potential from Subsection \ref{delta'-section}.
In a more detailed form, $f_n(x,y)=f_n^\pm(x,y)$ for $\pm y>0$ with
\[
	f_{n}^{+}(x,y):=e^{ikx-2y}\chi_{n}(x)\widetilde{\chi}_{n}(y),
	\quad
	f_{n}^{-}(x,y):=-e^{ikx+2y}\chi_{n}(x)\widetilde{\chi}_{n}(y).
\]
The support of $f_n$ is in the rectangle surrounded by the dashed line in Fig.~\ref{Fig3}, and from now on we assume that the number $a>0$ in the above constructions is sufficiently small, so that the rectangle
does not touch the other branches of $\Gamma$ for large enough $n$.

\begin{figure}[b]
	\centering
\begin{tabular}{lp{90mm}}	
	\parbox[c]{40mm}{\begin{tikzpicture}
		\coordinate (b) at (0.5,-1);
		\coordinate (c) at (2.8,1);
		\coordinate (d) at (1.05,-1.7);
		\draw [domain=0:0.65, line width=3.0pt] plot (\x,3.2*\x);
		\draw [dashed] (0.5,-1) rectangle ++(2.3,2);
		\draw [draw=black] (0.8,-0.7) rectangle ++(1.7,1.4);
		\draw [draw=black, fill=lightgray] (0.8,-0.7) rectangle ++(1.7,1.4);
		\draw (b) node[below] {\tiny ($n$,-$an$)};
		\draw (c) node[above] {\tiny ($2n$,$an$)};
		\filldraw[black] (b) circle (1pt);
		\filldraw[black] (c) circle (1pt);
		\draw [->,color=black, line width=0.5mm] (-0.5,0) -- (3.5,0);
		\draw [->,color=black, line width=0.5mm] (0,-1) -- (0,2);
		\draw [domain=0:3, line width=3.0pt] plot (\x,0*\x);
	\end{tikzpicture}}
	&
	\parbox[c]{85mm}{\caption{The product of the cut-off functions vanishes outside the dashed rectangle and is equal to one in the shaded area.}\label{Fig3}}
\end{tabular}
\end{figure}
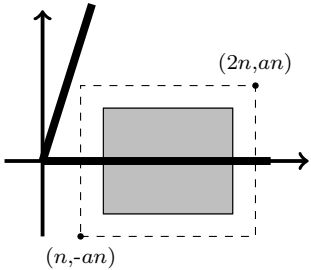

One has $f_n\in D(T_{\Gamma})$ with $(T_{\Gamma}f_n)(x,y)=-\Delta f_n^\pm(x,y)$ for $\pm y>0$ (see Appendix \ref{domainapp}). For $(x,y)\in (n+1,2n-1)\times (-an+1,an-1)$ there holds $|f_{n}(x,y)|=e^{-2|y|}$, hence,
for large $n$ one has
\begin{align*}
		\lVert f_{n}\rVert_{L^{2}(\mathbb{R}^{2})}^{2}&\ge\int_{n+1}^{2n-1}\int_{-an+1}^{an-1}
		e^{-4|y|}\dd y \dd x		\\
		&=2(n-2) \int_{0}^{an-1}e^{-4y}\dd y=2 (n-2) \, \dfrac{1-e^{4(1-an)}}{4}\geq Cn
\end{align*}
with a suitable $C>0$. We further compute
\begin{align*}
-\Delta f_{n}^{+}=(k^{2}-4)f^{+}_{n}&+(4\widetilde\chi_{n}'(y)-\widetilde\chi_{n}''(y))\chi_{n}(x)e^{ikx-2y}
\\
&-\big(2ik\chi'_{n}(x)+\chi''_{n}(x)\big)\widetilde{\chi}_{n}(y)e^{ikx-2y}.
\end{align*}
The support of $g_{n}:=-\Delta f_{n}^{+}-(k^{2}-4)f_{n}^{+}$ is contained in the set
\[
\Big(\big ([n,n+1]\cup[2n-1,2n]\big)\times[0,an] \Big)\cup\Big([n+1,2n-1]\times[an-1,an]\Big)
\]
and one has the pointwise estimate $|g_n(x,y)|\le c e^{-2|y|}$ with a suitable $c>0$. Hence, with a suitable $c_1>0$ and all large $n$,
\begin{equation*}
	\begin{aligned}
		\lVert-&\Delta f_{n}^{+}-(k^{2}-4)f_{n}^{+}\rVert^{2}_{L^{2}(\RR^2_+)}=\lVert g_n\rVert^{2}_{L^{2}(\RR^2_+)}
		\\
		&\leq c\left(\int_{[n,n+1]\cup[2n-1,2n]}\int_{0}^{an}e^{-4y}\dd y\dd x+\int_{n+1}^{2n-1}\int_{an-1}^{an}e^{-4y}\dd y\dd x\right)
		\\
		&=\frac{c}{4}\Big(2+\big((n-2)e^{4}-n\big)e^{-4an}\Big)\leq c_{1},
	\end{aligned}
\end{equation*}
and analogously one estimates $\lVert-\Delta f_{n}^- -(k^{2}-4)f_{n}^-\rVert^{2}_{L^{2}(\RR^2_-)}\le c_2$. Then, with $(T_\Gamma f_n)(x,y)=-\Delta f_n^\pm(x,y) \text{ for } \pm y>0$, we get
\begin{align*}
\dfrac{\|-\Delta f_{n}^{+}-(k^{2}-4)f_{n}^{+}\|^2_{L^{2}(\RR^2_+)}+\lVert-\Delta f_{n}^{-}-(k^{2}-4)f_{n}^{-}\rVert^{2}_{L^{2}(\RR^2_-)}}{\|f_{n}\|^2_{L^2(\RR^2)}}
\\
\equiv\dfrac{\Big\|\Big(T_\Gamma-(k^2-4)\Big)f_n\Big\|^2_{L^2(\RR^2)}}{\|f_{n}\|^2_{L^2(\RR^2)}}\le \dfrac{c_1+c_2}{Cn}\xrightarrow{n\to\infty}0,
\end{align*}
which shows $k^2-4\in\spec T_\Gamma$. As $k\in\RR$ is arbitrary, one obtains the inclusion $[-4,\infty)\subset\spec T_\Gamma$. As $[-4,\infty)$ has no isolated points, one even has $[-4,\infty)\subset\specess T_\Gamma$.
\end{proof}

In order to continue we compare $T_\Gamma$ with the Robin Laplacians in infinite sectors.
Assume first that $M\ge 2$, consider the sectors $V_j$ defined in \eqref{eqvj}, and for $\gamma>0$
define the closed, lower semibounded, symmetric sesquilinear forms $r_j^\gamma$ in $L^2(V_j)$ by
\begin{equation}
r_j^\gamma(u,u)=\int_{V_j}|\nabla u|^2\dd x-\gamma\int_{\partial V_j} |u|^2\dd\sigma,
\quad D(r_j^\gamma)=H^1(V_j),
\end{equation}
and the associated self-adjoint operators $R^\gamma_j$ in $L^2(V_j)$. Remark that for any $j$ one has
$\partial V_j=\Gamma_j\cup\Gamma_{j+1}$.
Further consider the unitary map
\begin{equation}
\label{eqjj}
J:L^2(\RR^2)\ni u\mapsto (u_1,\dots,u_M)\in \bigoplus_{j=1}^M L^2(V_j),
\quad
u_j:=u|_{V_j}.
\end{equation}
\begin{lemma}\label{lem2}
	For any star graph $\Gamma$ with $M\ge 2$ branches there holds
\[
T_\Gamma\ge \bigoplus_{j=1}^M R^2_j \text{ using } J
\]
with the relation $\ge$ defined as in Section \ref{sec-min-max}.
\end{lemma}

\begin{proof}
Using the notation \eqref{eqjj}, for any $u\in H^1(\RR^2\setminus\Gamma)$ one obtains with the help of the triangle inequality
\begin{align*}
	\int_\Gamma \big|[u]\big|^2\dd\sigma&=\sum_{j=1}^M \int_{\Gamma_j} |u_j-u_{j-1}|^2\dd\sigma\le 2 \sum_{j=1}^M \int_{\Gamma_j} \big( |u_j|^2+ |u_{j-1}|^2\big)\dd\sigma.
\end{align*}
Hence, for the sesquilinear form $t_\Gamma$ for $T_\Gamma$,
\begin{align*}
t_\Gamma(u,u)=\int_{\RR^2\setminus \Gamma} |\nabla u|^2\dd x - \int_\Gamma \big|[u]\big|^2\dd\sigma,\quad
D(t_\Gamma)=H^1(\RR^2\setminus \Gamma),
\end{align*}
one has
\begin{align*}
t_\Gamma(u,u)&=\sum_{j=1}^M \int_{V_j}|\nabla u|^2	\dd x -\int_\Gamma \big|[u]\big|^2\dd\sigma\\
&\ge \sum_{j=1}^M \int_{V_j}|\nabla u|^2\dd x - 2 \sum_{j=1}^M \int_{\Gamma_j} \big( |u_j|^2+ |u_{j-1}|^2\big)\dd\sigma\\
&=\sum_{j=1}^M  \bigg(\int_{V_j}|\nabla u_j|^2\dd x -2 \int_{\Gamma_j}|u_j|^2\dd\sigma -2 \int_{\Gamma_{j+1}}|u_j|^2\dd\sigma
\bigg)\\
&=\sum_{j=1}^M  \bigg(\int_{V_j}|\nabla u_j|^2\dd x -2 \int_{\partial V_j}|u_j|^2\dd\sigma\bigg)=\sum_{j=1}^M r^2_j(u_j,u_j),
\end{align*}
and the last expression is exactly the sesquilinear form for  $\bigoplus_{j=1}^M R^2_j$.
\end{proof}

\begin{lemma}\label{lem3}
For any star graph $\Gamma$ there holds $\specess T_\Gamma\subset[-4,\infty)$.
\end{lemma}

\begin{proof}
We need to show the inequality $\inf\specess T_\Gamma\ge -4$.

Consider first the case of $\Gamma$ with $M\ge 2$ branches. Each sector $V_j$  in \eqref{eqvj} is a rotated copy of the sector $U_{\varphi_j}$ from \eqref{equt}, and it follows
by standard arguments that $R^\gamma_j$ is unitarily equivalent to the standard Robin Laplacian $Q^\gamma_{\varphi_j}$ from Subsection~\ref{sec-robin}, which implies $\specess R^\gamma_j=\specess Q^\gamma_{\varphi_j}=[-\gamma^2,\infty)$.
By Lemma \ref{lem2} and the min-max principle (Subsection \ref{sec-min-max}) there holds
\[
\inf \specess T_\Gamma\ge \inf\specess \bigoplus_{j=1}^M R^2_j\equiv \min_{j} \inf \specess R^2_j=\min_j (-2^2)=-4.
\]

If $\Gamma$ only contains one branch, i.e. if $\Gamma$ is just a half-line starting at the origin,
we extend it to a line $\Tilde\Gamma$, then $T_\Gamma\ge T_{\Tilde \Gamma}$ using the identity map (Lemma \ref{lem0})
implying $\inf\specess T_\Gamma\ge \inf\specess T_{\Tilde \Gamma}=-4$.
\end{proof}

By combining Lemmas \ref{lem1} and \ref{lem3} one arrives at
\begin{corollary}
	For any star graph $\Gamma$ there holds $\specess T_\Gamma=[-4,\infty)$.
\end{corollary}

\section{Discrete eigenvalues for star graphs}\label{secdisc}

\subsection{Cardinality of the discrete spectrum}

We now discuss the existence and non-existence of discrete eigenvalues of $T_{\Gamma,\alpha}$. Again, due to the unitary equivalence \eqref{unit1} it is sufficient to consider the operator $T_\Gamma$ only.
We first consider the degenerate cases.
\begin{lemma}\label{lem4}
If $\Gamma$ is a half-line or a line, then $T_\Gamma$	has no discrete spectrum.
\end{lemma}

\begin{proof}
Due to $\specess T_\Gamma=[-4,\infty)$ it is sufficient to show $\inf\spec T_\Gamma\ge -4$.

	If $\Gamma$ is a line, without loss of generality one may assume that $\Gamma$ coincides with the $x$-axis, then the sesquilinear form $t_\Gamma$ for $T_\Gamma$ is
	\[
	t_\Gamma(u,u)=\int_{\RR\times (\RR\setminus \{0\})} |\nabla u|^2\dd x\dd y - \int_\RR \big|u(x,0^+)-u(x,0^-)\big|^2\dd x
	\]
	with $D(t_\Gamma)=H^1\big(\RR\times (\RR\setminus \{0\})\big)$. As discussed in Subsection \ref{delta'-section}, for any $f\in H^1(\RR\setminus \{0\})$ one has
\[
\int_{\mathbb{R}}\lvert f'(x)\rvert^{2}\dd x-\big\lvert f(0+)-f(0-)\big\rvert^{2}\ge -4 \|f\|^2_{L^2(\RR)},
\]
which implies by Fubini's theorem $t_\Gamma(u,u)\ge -4\|u\|^2_{L^2(\RR^2)}$ for all $u\in D(t_\Gamma)$,
i.e. $\inf\spec T_\Gamma\ge -4$.

If $\Gamma$ is a half-line, we extend it to a line $ \Tilde \Gamma$ and remark again that
$T_\Gamma\ge T_{\Tilde\Gamma}$ using the identity map (Lemma \ref{lem0}). By the first part of the proof
we have $\inf\spec T_{\Tilde\Gamma}\ge -4$, and the min-max principle implies
\[
\inf\spec T_{\Gamma}=\Lambda_1(T_\Gamma)\ge \Lambda_1(T_{\Tilde\Gamma})=\inf\spec T_{\Tilde\Gamma}\ge -4, \qedhere
\]
\end{proof}

Now we want to show that the discrete spectrum of $T_\Gamma$ is non-empty for all star graphs not covered by Lemma \ref{lem4}. We consider first the case of two branches, i.e. the operator $H_\theta$ defined in \eqref{eqhtt}.

\begin{lemma}\label{lem9}
For any $\theta\in(0,\frac{\pi}{2})$ the discrete spectrum of $H_\theta$ is non-empty.
\end{lemma}

\begin{proof}
Denote by $\Omega_\pm$ the two connected components of $\RR^2\setminus\Gamma_\theta$,
\begin{equation}
	\label{eqompp}
	\Omega_+:=\{(x,y): |\arg(x+iy)|<\theta\},\quad
	\Omega_-:=\{(x,y): |\arg(x+iy)|>\theta\},
\end{equation}
see Fig.~\ref{fig1}(b), and consider the unitary map
\[
J:L^2(\RR^2)\ni u\mapsto \Tilde u:=(1_{\Omega_+} -1_{\Omega_-})u\in L^2(\RR^{2}\setminus\Gamma_{\theta}).
\]
For $u\in H^1(\RR^2)$ one has $Ju\in H^1(\RR^2\setminus \Gamma_\theta)$, and
\begin{align*}
	\int_{\RR^2}|\nabla u|^2\dd x&=\int_{\RR^2\setminus\Gamma_\theta}|\nabla u|^2\dd x=\int_{\RR^2\setminus\Gamma_\theta}|\nabla \Tilde u|^2\dd x,\\
	\int_{\Gamma_\theta}\big|[\Tilde u]\big|^2\dd\sigma&= \int_{\Gamma_\theta}|2u|^2\dd\sigma=4\int_{\Gamma_\theta}|u|^2\dd\sigma,
\end{align*}
implying
\begin{align*}
h_\theta(Ju,Ju)&=\int_{\RR^2\setminus\Gamma_\theta}|\nabla \Tilde u|^2\dd x-\int_{\Gamma_\theta}\big|[\Tilde u]\big|^2\dd\sigma\\
&=\int_{\RR^2}|\nabla u|^2\dd x -4\int_{\Gamma_\theta}|u|^2\dd\sigma=a^4_\theta(u,u), \quad u\in D(a^4_{\theta});
\end{align*}
recall that the form $a^\gamma_\theta$ and the associated operator $A^\gamma_\theta$ were discussed in Subsection \ref{sec-delta}. The above shows that $H_\theta\le A^4_\theta$ using $J$, while the previous analysis
gives $\specess H_\theta=\specess A^4_\theta=[-4,\infty)$. As $A^4_\theta$ has at least one discrete eigenvalue
in $(-\infty,-4)$, the same holds for $H_\theta$ due to the min-max principle.
\end{proof}
We remark that the argument above is a particular case of a more general construction from \cite{bel14} on 
various relations between $\delta$ and $\delta'$-potentials on partitions of $\RR^n$, but
we preferred to reproduce it in order to have a self-contained presentation.

\begin{lemma}
	If the star graph $\Gamma$ is not a line or half-line, then the discrete spectrum of $T_\Gamma$ is non-empty.
\end{lemma}

\begin{proof}
	By assumption, the star graph $\Gamma$ contains a rotated copy $\Gamma'$ of $\Gamma_\theta$ with some $\theta\in(0,\frac{\pi}{2})$.
	As $T_{\Gamma'}$ is unitarily equivalent to $H_\theta$, it has at least one discrete eigenvalue in $(-\infty,-4)$ by Lemma \ref{lem9}.
	Furthermore, we have $T_{\Gamma}\le T_{\Gamma'}$ using the identity map (Lemma \ref{lem0}). As both operators have the same essential spectrum $[-4,\infty)$,
	it follows by the min-max principle that $T_\Gamma$ also has at least one discrete eigenvalue in $(-\infty,-4)$.
\end{proof}
\begin{lemma}
	For any star graph $\Gamma$ the discrete spectrum of $T_\Gamma$ is finite.
\end{lemma}

\begin{proof}
The statement is true for $M=1$ by Lemma \ref{lem4}, so we only need to consider $M\ge 2$.
We use again the relation $T_\Gamma\ge \bigoplus_{j=1}^M R^2_j$ using $J$ as discussed in Lemma \ref{lem2}.
Introducing again the angles $\varphi_j$ by \eqref{phij} we conclude that
each $R^2_j$ is unitarily equivalent to $Q^2_{\varphi_j}$ (Subsection \ref{sec-robin}) and has
$n_j<\infty$ discrete eigenvalues in $(-\infty,-4)$. Then $\bigoplus_{j=1}^M R^2_j$ has $n_1+\dots+n_M<\infty$
eigenvalues in $(-\infty,-4)$, and by the min-max principle the same holds for $T_\Gamma$.
\end{proof}

At this point all statements of Theorem \ref{thm1} (general star graphs) are completely proved.

\section{Eigenvalues of star graphs with two branches}\label{secangle}

Now we look more attentively at the case of two branches, i.e. the operator $H_{\theta,\alpha}$ from \eqref{eqhtt}.
In order to complete the proof of Theorem \ref{thm2} it remains to prove the statements 
concerning the dependence of the eigenvalues on $\theta$. As in the above cases, it is sufficient to prove the statements for $\alpha=-1$, i.e. for the operator $H_\theta$ only. Most steps are inspired by respective constructions in \cite{en,kp18}.

\subsection{Continuity and monotonicity}

We first decompose the operator $H_\theta$ using the parity with respect to $y$. Namely,
consider the maps
\begin{gather*}
\Pi_\pm:L^{2}(\mathbb{R}^{2})\to  L^{2}(\mathbb{R}^{2}_{+}),\quad
\Pi_\pm u(x,y):=\frac{u(x,y)\pm u(x,-y)}{\sqrt{2}},
\end{gather*}
then one easily observes that
\[
\Phi: L^{2}(\mathbb{R}^{2})\ni u\mapsto(\Pi_+ u, \Pi_- u)\in L^{2}(\mathbb{R}^{2}_{+})\oplus L^{2}(\mathbb{R}^{2}_{+})
\]
is a unitary operator. A short direct computation shows that 
\[
h_{\theta}(u,u)=h_{\theta}^{N}(\Pi_+ u,\Pi_+u)+h_{\theta}^{D}(\Pi_- u, \Pi_- u) \text{ for any $u\in D(h_\theta)$,}
\]
where both $h^{N/D}_{\theta}$  are closed, lower semibounded, symmetric sesquilinear forms in $L^2(\RR^2_+)$ given
by the same expression,
\begin{equation*}
h^{N/D}_{\theta}(v,v)=\int_{\mathbb{R}^{2}_{+}\setminus \Gamma^+_\theta}|\nabla v|^2\dd x-\int_{\Gamma^+_\theta}|[v]|^2\dd\sigma,
\quad
\Gamma^{+}_{\theta}:=\Gamma_\theta\cap \RR^2_+,
\end{equation*}
but with distinct domains,
\begin{align*}
	D(h^{N}_{\theta})&=\Pi_{+} D(h_{\theta})=H^{1}(\mathbb{R}^{2}_{+}\setminus\Gamma_{\theta}^{+}),\\ D(h^{D}_{\theta})&=\Pi_{-} D(h_{\theta})=\{v\in H^{1}(\mathbb{R}^{2}_{+}\setminus\Gamma_{\theta}^{+}):v(\cdot,0)=0\},
\end{align*}
where $v(\cdot,0)$ is understood in the sense of Sobolev traces,
and we denote by $H^{N/D}_\theta$ the associated self-adjoint operators in $L^2(\RR^2_+)$.
Consider the the closed, lower semibounded, symmetric sesquilinear form $\widetilde{h}_\theta$ in $L^2(\RR^2_+)\times L^2(\RR^2_+)$,
\[
\widetilde{h}_\theta\big((v,u),(v,u)\big):=h_{\theta}^{N}(v,v)+h_{\theta}^{D}(u,u), \quad D(\widetilde{h}_\theta):=D(h_{\theta}^{N})\times D(h_{\theta}^{D}),
\]
then the associated operator is $\widetilde{H}_\theta = H^N_\theta\oplus H^D_\theta$. On the other hand, the above constructions
show that $D(\widetilde{h}_\theta)=\Phi D(h_\theta)$ and that $h_\theta(u,u)=\widetilde{h}_\theta(\Phi u,\Phi u)$ for any $u\in D(h_\theta)$,
which means that $H_\theta =\Phi^{-1}\widetilde{H}_\theta \Phi$. 
We conclude that $H_\theta=\Phi^{-1} (H^N_\theta\oplus H^D_\theta)\Phi$ (see Ch.~6 §1 in Kato's book \cite{kato}
for a detailed discussion of unitary equivalences with the help of forms).

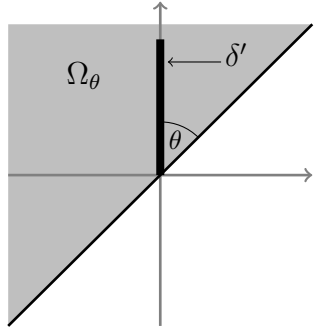
\begin{figure}[H]
	\centering
	
	\begin{tikzpicture}
		\coordinate (a) at (0.2,0.2);
		\coordinate (b) at (-2,2);
		\coordinate (c) at (1,1.3);
		\coordinate (d) at (-1,1);
		\coordinate (e) at (1,1);
		\coordinate (f) at (0,1);
		\fill[lightgray] (-2,-2)--(2,2)--(-2,2);
		\draw [->,color=gray, line width=1pt] (-2,0) -- (2,0);
		\draw [->,color=gray, line width=1pt] (0,-2) -- (0,2.3);
		\draw [domain=-2:2, line width=1pt] plot (\x,\x);    
		\draw [color=black, line width=3pt] (0,0) -- (0,1.8);
		
		\draw (0.5,0.5) arc (45:90:0.7);
		\draw [->] (0.8,1.5) -- (0.1,1.5);
		\draw (c) node[above] {$\delta'$};
		\draw (a) node[above] {\footnotesize $\theta$};
		\draw (d) node[above] {$\Omega_{\theta}$};
		
	\end{tikzpicture}
	\caption{Visualization of  $\Omega_{\theta}$.}\label{fig4}
	
\end{figure}

It will be useful to apply an additional rotation by $\frac{\pi}{2}-\theta$, which bijectively maps $\Gamma_\theta^+$ on the positive $y$-half-axis $Oy^+$ and $\RR^2_+$ on the set (see Fig.~\ref{fig4})
\[
\Omega_{\theta}=\{(x,y)\in\mathbb{R}^{2}:x<y\tan \theta\}.
\]
The corresponding coordinate change transforms $H^{N/D}_\theta$ to unitarily equivalent self-adjoint operators $\Tilde H^{N/D}_\theta$ in $L^2(\Omega_\theta)$ whose sesquilinear forms are
\begin{align*}
	\Tilde h^{N/D}_{\theta}(v,v)&=\int_{\Omega_\theta \setminus Oy^+}|\nabla v|^2\dd x \dd y-\int_0^\infty|v(0^+,y)-v(0^-,y)|^2\dd y,\\
	D(\Tilde h^{N}_{\theta})&=H^1(\Omega_\theta \setminus Oy^+),\quad
	D(\Tilde h^{D}_{\theta})=\{v\in H^1(\Omega_\theta \setminus Oy^+):\, v=0 \text{ on } \partial\Omega_\theta\}.
\end{align*}

We summarize the above constructions as follows:
\begin{lemma}
The operator $H_\theta$ is unitarily equivalent to $\Tilde H^{N}_\theta \oplus \Tilde H^{D}_\theta$.
\end{lemma}

Recall that we are interested in the discrete spectrum of $H_\theta$, which is located in $(-\infty,-4)$.

\begin{lemma}
	There holds $\spec \Tilde H^D_\theta\subset [-4,\infty)$.
\end{lemma}	

\begin{proof} For $v\in D(\Tilde h^D_\theta)$, let $\Tilde v$ denote its extension by zero to the whole of $\RR^2$, then $\Tilde v\in H^{1}(\RR^2\setminus Oy^{+})$, and
\begin{equation}
	\label{loc11}
\Tilde h^{D}_{\theta}(v,v)=\int_{\RR^2\setminus Oy^{+}} |\nabla \Tilde v|^2\dd x\dd y-\int_\RR|\Tilde v(0^+,y)-\Tilde v(0^-,y)|^2\dd y.
\end{equation}
By Fubini's theorem, for almost every $y\in\RR$ the function $x\mapsto \Tilde v(x,y)$ is in $H^1(\RR\setminus\{0\})$, hence, as discussed in Subsection \ref{sec-delta},
\[
\int_\RR |\partial_x \Tilde v(x,y)|^2\dd x-|\Tilde v(0^+,y)-\Tilde v(0^-,y)|^2\ge -4 \int_\RR |\Tilde v(x,y)|^2\dd x.
\]
The substitution into \eqref{loc11} shows that for any $v\in D(\Tilde h^D_\theta)$ one has
\[
\Tilde h^{D}_{\theta}(v,v)\ge -4 \|\Tilde v\|^2_{L^2(\RR^2)}\equiv -4 \|v\|^2_{L^2(\Omega_\theta)},
\]
and the min-max principle implies $\inf\spec \Tilde H^D_\theta\ge -4$.
\end{proof}

\begin{corollary}
The spectrum of $H_\theta$ in $(-\infty,-4)$ coincides with the spectrum of $\Tilde H^N_\theta$ in $(-\infty,-4)$.
\end{corollary}

Now the behavior of the discrete eigenvalues of $H_\theta$ is reduced to that
for $\Tilde H^N_\theta$.
In order to analyze $\Tilde H^N_\theta$ we consider another unitary map,
\[
U:L^{2}(\Omega_{\frac{\pi}{4}})\rightarrow L^{2}(\Omega_{\theta}),
\quad
Uv(x,y)=\sqrt{\tan \theta}v(x,y\tan \theta),
\]
which maps $D(\Tilde h^N_\frac{\pi}{4})$ bijectively onto $D(\Tilde h^N_\theta)$, then a short computation shows that for any $v\in D(\Tilde h^N_\frac{\pi}{4})$ there holds
\begin{equation*}
\begin{aligned}
\Tilde{h}^N_{\theta}(Uv,Uv)=&\int_{\Omega_{\frac{\pi}{4}}\setminus Oy^+}\big(|\partial_x v|^2+(\tan \theta)^2 |\partial_y v|^2\big)\dd x\dd y
\\
&-\int_{0}^{\infty}\lvert v(0^+,y)-v(0^-,y)\rvert^{2}\dd y=:s_{\theta}(v,v).
\end{aligned}
\end{equation*}
By construction, the expression $s_\theta$, with $D(s_\theta)=D(\Tilde h^N_\frac{\pi}{4})$, defines a closed, lower semibounded, symmetric sesquilinear form in $L^2(\Omega_\frac{\pi}{4})$, and the associated operator $S_\theta$ is unitarily equivalent to $\Tilde H^N_\theta$, which proves
\begin{corollary}\label{corol15}
	The spectrum of $H_\theta$ in $(-\infty,-4)$ coincides with the spectrum of $S_\theta$ in $(-\infty,-4)$.
\end{corollary}

\begin{lemma}\label{lem16}
For all fixed $n\in\mathbb{N}$ the map $(0,\frac{\pi}{2})\ni\theta\mapsto \Lambda_{n}(S_{\theta})$ is non-decreasing and continuous.
\end{lemma}

\begin{proof}
The analysis of $S_\theta$ is simpler as the angle $\theta$ only appears in its coefficients and not in the domain.

Remark that the domains of all sesquilinear forms $s_\theta$ are the same.
For $\theta_1\le \theta_2$	we have $\tan\theta_1\le \tan\theta_2$, which implies
that $S_{\theta_1}\le S_{\theta_2}$	using the identity map (see Subsection \ref{sec-min-max}).
Hence, $\Lambda_n(S_{\theta_1})\le \Lambda_n(S_{\theta_2})$, which shows the monotonicity.
	
For the continuity, let $\theta\in(0,\frac{\pi}{2})$, then
\[
\|v\|^2_s:=s_\theta(v,v)+\big(1-\Lambda_1(S_\theta)\big)\|v\|^2_{L^2(\Omega_{\frac{\pi}{4}})}
\]
defines an equivalent norm on $H^1(\Omega_{\frac{\pi}{4}}\setminus Oy^+)$, and there is $c>0$ with 
\[
\int_{\Omega_{\frac{\pi}{4}}\setminus Oy^+} |\partial_y v|^2\dd x\dd y\le \|v\|_{H^{1}(\Omega_{\frac{\pi}{4}}\setminus Oy^+)}^{2}\le c \|v\|^2_s \text{ for all } v\in D(s_\theta).
\]
Let $n\in\NN$ be fixed. Pick $\varepsilon>0$ and choose $\delta>0$ with
$|\tan^2 \theta'-\tan^2\theta|<\varepsilon$ for all $\theta'$ with $|\theta'-\theta|<\delta$,
then for any such $\theta'$ and any $v\in D(s_\theta)\equiv D(s_{\theta'})$ there holds
\begin{align*}
\big|s_{\theta'}(v,v)-s_{\theta}(v,v)\big|&=
\big|\tan^2 \theta'-\tan^2\theta\big|\int_{\Omega_{\frac{\pi}{4}}\setminus Oy^+}|\partial_y v|^2\dd x\dd y\\
 &\le \varepsilon \int_{\Omega_{\frac{\pi}{4}}\setminus Oy^+}|\partial_y v|^2\dd x\dd y\le c\varepsilon \|v\|^2_s.
\end{align*}
Hence,
\[
s_{\theta^{'}}(v,v)\ge s_{\theta}(v,v)-c\varepsilon \|v\|^2_s=(1-c\varepsilon)s_\theta(v,v)-c\varepsilon \big(1-\Lambda_1(S_\theta)\big)\|v\|^2_{L^2(\Omega_{\frac{\pi}{4}})}.
\]
The expression on the right-hand side is the sesquilinear form for the self-adjoint operator $(1-c\varepsilon)S_\theta - c\varepsilon\big(1-\Lambda_1(S_\theta)\big)$, and the min-max principle implies
the inequality $\Lambda_n(S_{\theta'})\ge (1-c\varepsilon)\Lambda_n(S_\theta)-c\varepsilon\big(1-\Lambda_1(S_\theta)\big)$.
Analogously one obtains the upper bound $\Lambda_n(S_{\theta'})\le (1+c\varepsilon)\Lambda_n(S_\theta)+c\varepsilon\big(1-\Lambda_1(S_\theta)\big)$,
so altogether one arrives at
\[
\big|\Lambda_n(S_{\theta'})-\Lambda_n(S_\theta)\big|\le c\varepsilon \Big(|\Lambda_n(S_\theta)|
+|\Lambda_1(S_\theta)|+1\Big)  \text{ as } |\theta'-\theta|<\delta.
\]
As $\varepsilon>0$ can be chosen arbitrarily small, this shows the continuity.
\end{proof}

\begin{corollary}
Let $n\in\mathbb{N}$ and $\theta_{n}\in(0,\frac{\pi}{2})$ be such that $H_{\theta_{n}}$ has at least $n$ discrete eigenvalues in $(-\infty,-4)$.
Then $H_{\theta}$ has at least $n$ discrete eigenvalues for all $\theta\in(0,\theta_{n})$, and the function $(0,\theta_{n})\ni\theta\mapsto E_{n}(H_{\theta})$ is strictly increasing and continuous.
\end{corollary}
\begin{proof}
By Corollary \ref{corol15} we have $\Lambda_n(H_\theta)<-4$ iff $\Lambda_n(S_\theta)<-4$, and then $\Lambda_n(S_\theta)=\Lambda_n(H_\theta)$ for all such $n$.
	
The initial assumption means $\Lambda_n(H_{\theta_n})<-4$, i.e. $\Lambda_n(S_{\theta_n})<-4$. By Lemma \ref{lem16} it follows that $\Lambda_n(S_\theta)\le \Lambda_n(S_{\theta_n})<-4$ for all $\theta\in(0,\theta_n)$,
hence $\Lambda_n(H_\theta)=\Lambda_n(S_\theta)<-4=\inf\specess H_\theta$ for the same $\theta$, showing that $H_\theta$ has at least $n$ eigenvalues $E_j(H_\theta)=\Lambda_j(H_\theta)=\Lambda_j(S_\theta)$, $j=1,\dots,n$. The non-strict monotonicity and continuity follow from Lemma \ref{lem16}.
It remains to show the strict monotonicity of $\theta\mapsto \Lambda_n(S_\theta)\equiv E_n(S_\theta)$.

Let $\theta,\theta'\in(0,\theta_{n})$ with $\theta<\theta'$, then
$\kappa:=\tan^{2}\theta'-\tan^{2}\theta>0$,
and for any $u\in D(s_{\theta})\equiv D(s_{\theta'})$ there holds
\begin{equation*}
s_{\theta}(v,v)=s_{\theta'}(v,v)-\kappa\int_{\Omega_{\frac{\pi}{4}}}|\partial_y v|^2\dd x\dd y.
\end{equation*}
Let $w_{1},...,w_{n}$ be an orthonormal family of eigenvectors of $S_{\theta'}$
for the eigenvalues $E_1(S_{\theta'}),\dots,E_n(S_{\theta'})$ and $W_n:=\vspan(w_1,\dots,w_n)$, then for any $v\in W_n$ one has $s_{\theta'}(v,v)\le E_n(S_{\theta'}) \|v\|^2_{L^2(\Omega_{\frac{\pi}{4}})}$.
As $W_n$ is an $n$-dimensional subspace of $D(s_\theta)$, due to the min-max principle we have
\begin{align*}
	E_n(S_\theta)&\le \sup\big\{s_\theta(v,v): v\in W_n,\,\|v\|^2_{L^2(\Omega_{\frac{\pi}{4}})}=1\big\}\\
	&=\sup\Big\{s_{\theta'}(v,v)-\kappa\int_{\Omega_{\frac{\pi}{4}}}|\partial_y v|^2\dd x\dd y: v\in W_n,\,\|v\|^2_{L^2(\Omega_{\frac{\pi}{4}})}=1\Big\}\\
	&\le \sup\Big\{s_{\theta'}(v,v): v\in W_n,\,\|v\|^2_{L^2(\Omega_{\frac{\pi}{4}})}=1\Big\}\\
	&\quad
	-\kappa \inf\Big\{\int_{\Omega_{\frac{\pi}{4}}}|\partial_y v|^2\dd x\dd y: v\in W_n,\,\|v\|^2_{L^2(\Omega_{\frac{\pi}{4}})}=1\Big\}\\
	&\le E_n(S_{\theta'})-\kappa a,\\
	\text{with }a&:=\inf\Big\{\int_{\Omega_{\frac{\pi}{4}}}|\partial_y v|^2\dd x\dd y: v\in W_n,\,\|v\|^2_{L^2(\Omega_{\frac{\pi}{4}})}=1\Big\}.
\end{align*}
Clearly, $a\ge 0$. Assume that $a=0$. As the unit sphere of $W_{n}$ is compact, the infimum in $a$ is attained at some $v\in W_n$ with $\|v\|^2_{L^2(\Omega_{\frac{\pi}{4}})}=1$, and then
\begin{equation*}
\int_{\Omega_{\frac{\pi}{4}}}\lvert\partial_{y}v\rvert^{2}\dd x \dd y=0,
\end{equation*}
which implies that $\partial_y v=0$. So $v$ takes the form $v(x,y)=w(x)$ for some function $w$, i.e. $v$ only depends on the first variable.
Since for any $b\in\mathbb{R}$ we have $(-\infty,b)\times (b,\infty)\subset\Omega_{\frac{\pi}{4}}$, there has to hold
\begin{equation*}
\begin{aligned}
\infty>\int_{\Omega_{\frac{\pi}{4}}}|v|^2\dd x\dd y&\geq\int_{b}^{\infty}\int_{-\infty}^{b}|v|^2\dd x\dd y
\\
&=\int_{b}^{\infty}\int_{-\infty}^{b}|w(x)|^2\dd x\dd y \geq\int_{b}^{\infty}\|w\|_{L^{2}(-\infty,b)}^2\dd y.
\end{aligned}
\end{equation*}
But this shows that $\|w\|_{L^{2}(-\infty,b)}^2=0$ for any $b\in\RR$, i.e. $w=0$.
But then $v=0$, which contradicts the normalization of $v$. 
Hence, $a=0$ is impossible: one has $a>0$
and  $E_{n}(S_{\theta})\le E_{n}(S_{\theta'})-\kappa a<E_{n}(S_{\theta'})$.
\end{proof}

\subsection{Small angle asymptotics} Now we prove the remaining statement concerning the behavior of eigenvalues as the angle $\theta$ becomes small.
\begin{lemma}\label{lem-angle}
For any fixed $n\in\mathbb{N}$ there exists $\theta_n>0$ such that $H_\theta$ has at least $n$ discrete eigenvalues
for all $\theta\in(0,\theta_n)$, and there holds
\begin{equation}
	   \label{eqent}
E_{n}(H_{\theta})=-\frac{1}{(2n-1)^{2}\theta^{2}}+O\Big(\dfrac{1}{\theta}\Big)\text{ as }
\theta\rightarrow0^+.
\end{equation}
\end{lemma}
\begin{proof}
We again decompose $\RR^2\setminus\Gamma_\theta$ in $\Omega_\pm$ as in \eqref{eqompp},
	and for $u\in L^2(\RR^2)$ denote $u_\pm:=u|_{\Omega_\pm}$, then the expression for the sesquilinear form $h_\theta$
	can be rewritten as
\begin{equation}
	\label{ht2}
	h_\theta(u,u)=\int_{\Omega_+} |\nabla u_+|^2\dd x+\int_{\Omega_-} |\nabla u_-|^2\dd x
	-\int_{\Gamma_\theta} |u_+-u_-|^2\dd\sigma.
\end{equation}

Consider the linear isometry
\begin{equation*}
J: L^2(\Omega_+)\to L^2(\RR^2),
\quad
		Jv(x)=
		\begin{cases}
			v(x),&x\in\Omega_+,\\
			0, &\text{otherwise,}
		\end{cases}
\end{equation*}
then $J H^1(\Omega_+)\subset H^1(\RR^2\setminus \Gamma_\theta)\equiv D(h_\theta)$,
and for any $v\in H^1(\Omega_+)$ there holds
\begin{equation}\label{asymptEV2}
	\begin{aligned}
		h_{\theta}(Jv,Jv)&=\int_{\Omega_+}|\nabla v|^2\dd x-\int_{\Gamma_\theta}|v|^2\dd\sigma\\
		&\equiv \int_{\Omega_+}|\nabla v|^2\dd x-\int_{\partial\Omega_+}|v|^2\dd\sigma\equiv q_\theta(v,v);
\end{aligned}
\end{equation}
recall that the form  $q_\theta$ and the respective Robin Laplacian $Q_\theta$ are defined in \eqref{sec-robin}. Therefore, we have $H_\theta\le Q^1_\theta$ using $J$, which implies
\[
\Lambda_n(H_\theta)\le \Lambda_n(Q_\theta) \text{ for all  $n\in\NN$ and $\theta\in \Big(0,\tfrac{\pi}{2}\Big)$.}
\]
As discussed in Subsection \ref{sec-robin}, for any fixed $n\in\NN$ there holds
\[
 E_n(Q_\theta)\equiv \Lambda_n(Q_\theta)=-\dfrac{1}{(2n-1)^2\theta^2}+O(1) \text{ for } \theta\to 0^+,
\]
which implies the upper bound
\[
\Lambda_n(H_\theta)\le -\dfrac{1}{(2n-1)^2\theta^2}+O(1) \text{ for } \theta\to 0^+.
\]
For small $\theta$ the right-hand side is clearly smaller than $\inf\specess H_\theta\equiv -4$, which implies $E_n(H_\theta)=\Lambda_n(H_\theta)$ and shows the sought upper bound for $E_n(H_\theta)$ in \eqref{eqent}.

In order to obtain a lower bound for $\Lambda_n(H_\theta)$
we recall first that for arbitrary $a,b\ge 0$ and $\varepsilon>0$ one has
\[
2ab=2\cdot \sqrt{\varepsilon} a\cdot \frac{b}{\sqrt{\varepsilon}}\leq \varepsilon a^{2}+\dfrac{b^{2}}{\varepsilon}.
\]
Hence, for arbitrary $u\in D(h_\theta)$ one has
\begin{equation*}
\begin{aligned}
\lvert u_{+}-u_{-}\rvert^{2}\leq\lvert u_{+}\rvert^{2}+\lvert u_{-}\rvert^{2}+2\lvert u_{+}u_{-}\rvert\leq(1+\varepsilon)\lvert u_{+}\rvert^{2}+\big(1+\frac{1}{\varepsilon}\big)\lvert u_{-}\rvert^{2},
\end{aligned}
\end{equation*}
and the substitution into \eqref{ht2} gives
\begin{align*}
	h_\theta(u,u)&\ge \int_{\Omega_+} |\nabla u_+|^2\dd x+\int_{\Omega_-} |\nabla u_-|^2\dd x\\
	&\quad
	-(1+\varepsilon)\int_{\Gamma_\theta} |u_+|^2\dd\sigma-\big(1+\frac{1}{\varepsilon}\big)\int_{\Gamma_\theta} |u_-|^2\dd\sigma\\
	&=q^{1+\varepsilon}_\theta(u_+,u_+)-r(u_-,u_-),
\end{align*}
where $q^{1+\varepsilon}_\theta$ and the associated operator $Q^{1+\varepsilon}_\theta$ are discussed in Subsection \ref{sec-robin}, and $r$ is the closed, lower semibounded, symmetric sesquilinear form
\[
r(v,v)=\int_{\Omega_-} |\nabla v|^2\dd x-\big(1+\frac{1}{\varepsilon}\big)\int_{\Gamma_\theta} |v|^2\dd\sigma,
\quad D(r)=H^1(\Omega_-).
\]
Let $R$ be the self-adjoint operator in $L^2(\Omega_-)$ associated with $r$, then the above estimates mean that $H_\theta\ge Q^{1+\varepsilon}_\theta\oplus R$ using the unitary map
\[
J': L^2(\RR^2)\ni u\mapsto(u_+,u_-)\in L^2(\Omega_+)\oplus L^2(\Omega_-),
\]
which implies $\Lambda_n(H_\theta)\ge \Lambda_n(Q^{1+\varepsilon}_\theta\oplus R)$ for any $n$ and any $\theta$. Remark that $R$ is just a rotated version of $Q^{1+\frac{1}{\varepsilon}}_{\pi-\theta}$.
Since, by the min-max principle, $\Lambda_1(Q^{1+\frac{1}{\varepsilon}}_{\pi-\theta})=\inf \spec Q^{1+\frac{1}{\varepsilon}}_{\pi-\theta}$ and $\pi-\theta\geq\frac{\pi}{2}$, we know from Subsection \ref{sec-robin} that $\Lambda_1(Q^{1+\frac{1}{\varepsilon}}_{\pi-\theta})= -(1+\frac{1}{\varepsilon})^2$. It follows that for any $n\in\NN$ and any $\theta\in(0,\frac{\pi}{2})$
one has
\begin{equation}
	  \label{lam3}
\begin{aligned}
\Lambda_n(H_\theta)&\ge \Lambda_n(Q^{1+\varepsilon}_\theta\oplus R)=
\Lambda_n(Q^{1+\varepsilon}_\theta\oplus Q^{1+\frac{1}{\varepsilon}}_{\pi-\theta})\\
&\ge \min\Big\{ \Lambda_n(Q^{1+\varepsilon}_\theta), -\Big(1+\frac{1}{\varepsilon}\Big)^2\Big\}.
\end{aligned}
\end{equation}
Now let $n\in\NN$ be fixed, then one has, as $\theta\to 0^+$,
\begin{align*}
\Lambda_n(Q^{1+\varepsilon}_\theta)&=(1+\varepsilon)^2
\Lambda_n(Q^{1}_\theta)= (1+\varepsilon)^2\Big[ -\dfrac{1}{(2n-1)^2\theta^2}+O(1)\Big],
\end{align*}
where the $O(1)$-term is independent of $\varepsilon$.
We now set $\varepsilon:=b_n \theta$ with $b_n>0$ that will be chosen later, then the preceding estimate implies
\[
\Lambda_n(Q^{1+\varepsilon}_\theta)= -\dfrac{1}{(2n-1)^2\theta^2}+O\Big(\dfrac{1}{\theta}\Big) \text{ as } \theta\to 0^+.
\]
At the same time, choosing $b_n\in(0,2n-1)$ we arrive at
\[
-\Big(1+\frac{1}{\varepsilon}\Big)^2=-1-\frac{2}{b_{n}\theta}-\frac{1}{b_{n}^{2}\theta^{2}}> \Lambda_n(Q^{1+\varepsilon}_\theta) \text{ as } \theta\to 0^+.
\]
By \eqref{lam3} it follows that for $\theta\to 0^+$ there holds
\begin{align*}
\Lambda_n(H_\theta)&\ge \min\Big\{ \Lambda_n(Q^{1+\varepsilon}_\theta), -\Big(1+\frac{1}{\varepsilon}\Big)^2\Big\}=\Lambda_n(Q^{1+\varepsilon}_\theta)=-\dfrac{1}{(2n-1)^2\theta^2}+O\Big(\dfrac{1}{\theta}\Big).
\end{align*}
We already know from the first part of the proof that $E_n(H_\theta)=\Lambda_n(H_\theta)$ for small $\theta$. Hence, the last inequality gives the sought lower bound for $E_n(H_\theta)$ in \eqref{eqent} and completes the proof.
\end{proof}

\begin{corollary}\label{corolast}
	Let $\alpha<0$.	For any $M\ge 2$ and $n\in\NN$ there exists a star graph $\Gamma$ such that $T_{\Gamma,\alpha}$ has at least 	$n$ discrete eigenvalues.
\end{corollary}
	
\begin{proof}
Take $\theta>0$ sufficiently small such that $H_\theta$ has $n$ discrete eigenvalues (which is possible by Lemma \ref{lem-angle}). Construct $\Gamma$ by adding arbitrary $M-2$ additional branches to $\Gamma_\theta$,
then $T_{\Gamma,\alpha}\le H_{\theta,\alpha}$ using the identity map (Lemma~\ref{lem0}). As both operators have the same essential spectrum $[-4\alpha^2,+\infty)$, it follows by the min-max principle
that $T_{\Gamma,\alpha}$ has at least $n$ discrete eigenvalues in $(-\infty,-4\alpha^2)$.
\end{proof}

\appendix

\section{Proof of $\spec T_{\Gamma,\alpha}=[0,\infty)$ for $\alpha\geq0$}\label{specapp}
In this section we want to show that $\spec T_{\Gamma,\alpha}=[0,\infty)$ in the case of a  ``repulsive interaction'' ($\alpha\geq0$).

Since $t_{\Gamma,\alpha}(u,u)\geq 0$ for all $u\in H^1(\mathbb{R}^2\setminus\Gamma)$, we have $T_{\Gamma,\alpha}\geq 0$
and $\spec T_{\Gamma,\alpha}\subset [0,\infty)$. For the reverse inclusion we make use of Weyl sequences. The set  $\RR^2\setminus\Gamma$ contains arbitrary large balls, and
for any $n\in\NN$ there exists $a_n\in\mathbb{R}^2\setminus\Gamma$ such that
$B_n:=\{x\in\RR^2: |x-a_n|<n\}\subset\mathbb{R}^2\setminus\Gamma$.
Let $k\in\RR$. Pick   $\chi\in C^\infty(\mathbb{R})$ with $0\leq \chi\leq 1$, $\chi(t)=1$ for $t\leq0$, $\chi(t)=0$ for $t\geq \frac{1}{2}$ and consider the functions
\[
f_n:x=(x_1,x_2)\mapsto e^{ikx_1}\chi\big(|x-a_n|-(n-1)\big)\in C_c^{\infty}(\mathbb{R}^2\setminus\Gamma).
\]
One has $f_n \subset C^\infty_c(\RR^2\setminus\Gamma)\subset D(T_{\Gamma,\alpha})$ with
$T_{\Gamma,\alpha} f_n=-\Delta f_n$. Due to $|f_n(x)|=1$ for $|x-a_n|\le n-1$ one obtains
\begin{equation}
	  \label{fnorm}
\|f_n\|^2_{L^2(\RR^2)}\ge \pi (n-1)^2.
\end{equation}
Using the abbreviation $(\dots):=\big(|x-a_n|-(n-1)\big)$ we compute
\[
-\Delta f_n=k^2f_n-2\nabla e^{ikx_1}\cdot \nabla\chi(\dots)-e^{ikx_1}\Delta\chi(\dots),
\]
which leads to $(T_{\Gamma,\alpha}-k^2)f_n=-2\nabla e^{ikx_1}\cdot \nabla\chi(\dots)-e^{ikx_1}\Delta\chi(\dots)$.
From this expression one immediately sees that
\begin{align*}
\|(T_{\Gamma,\alpha}-k^2)f_n\|_\infty&\le C:= 2k\|\nabla\chi\|_\infty+\|\Delta\chi\|_\infty,\\
\supp (T_{\Gamma,\alpha}-k^2)f_n&\subset \{x\in\RR^2: n-1\le|x-a_n|\le n-\frac12\},
\end{align*}
which implies $\|(T_{\Gamma,\alpha}-k^2)f_n\|^2_{L^2(\RR^2)}\le C^2 \pi \big( (n-\frac12)^2- (n-1)^2\big )\equiv C^2\pi\big(n-\frac{3}{4}\big)$. By combining this with \eqref{fnorm} one arrives at
\begin{equation*}
\begin{aligned}
\frac{\|(T_{\Gamma,\alpha}-k^2) f_n\|^2_{L^2(\mathbb{R}^2)}}{\|f_n\|^2_{L^2(\mathbb{R}^2)}}\leq\frac{C^2\pi\big(n-\frac{3}{4}\big)}{\pi (n-1)^2}\xrightarrow{n\to\infty} 0,
\end{aligned}
\end{equation*}
which shows $k^2\in\spec T_{\Gamma,\alpha}$. As $k\in\mathbb{R}$ is arbitrary, one obtains the inclusion $[0,\infty)\subset\spec T_{\Gamma,\alpha}$.

Therefore, we have proved that $\spec T_{\Gamma,\alpha}=[0,\infty)$. As each non-isolated point of the spectrum belongs to the essential spectrum, we arrive at the equalities $\specess T_{\Gamma,\alpha}=\spec T_{\Gamma,\alpha}=[0,\infty)$ and
$\specd T_{\Gamma,\alpha}=\emptyset$.

\begin{remark}  The approach works in much more general situations. Namely, if $\Omega$ is an open set and $T$ is a linear operator in $L^2(\Omega)$
such that  (i) $\Omega$ contains a sequence of balls $B_{r_n}\subset\Omega $ with radii $r_n$ such that $\lim_{n\rightarrow\infty} r_n=\infty$, and (ii) for each $n$ one has the inclusion $C^\infty_c(B_{r_n})\subset D(T)$
and for each $\varphi \in C^\infty_c(B_{r_n})$ one has $T\varphi=-\Delta\varphi$, then the identical argument shows that $[0,\infty)\subset \spec T$.
\end{remark}

\section{Proof of $f_n\in D(T_{\Gamma})$}\label{domainapp}

We want to show that the functions $f_n$ of the Weyl sequence from \eqref{weylseq} belong to $D(T_\Gamma)$ if $n$ is sufficiently large. Recall that this is equivalent to the existence of suitable $F_n\in L^2(\RR^2)$ such that
$t_{\Gamma}(u,f_n)=\langle u,F_n\rangle_{L^2(\mathbb{R}^2)}$ for all $u\in H^1(\mathbb{R}^2\setminus\Gamma)$,
and in that case $F_n=T_\Gamma f_n$, see \cite[Ch. VI, \S 2]{kato}.

First remark that $f_n\in H^2(\RR_+^2\cup \RR^2_-)$ and that due to the the properties of the eigenfunctions $\psi_1$ we have
\begin{equation} 
	\label{eqff}	
\partial_y f_n^+(x,0)=\partial_y f_n^-(x,0), \quad
f_n^+(x,0)-f_n^-(x,0)=-\partial_y f^+_n(x,0), \quad x\in \RR.
\end{equation}
Recall that  $\supp f_n\subset [n,2n]\times[-an,an]$ and $a>0$ is chosen sufficiently small, e.g. $a<\min_j |\tan \varphi_j|$,
then  $\supp f_n$ only touches the branch of $\Gamma$ lying on the positive $x$-axis if  $n$ is large.
Let $\theta_n\in C^\infty_c(\RR^2)$ such that $\theta_n=1$ in $ [n-1,2n+1]\times [-an-1,an+1]$ and $\supp \theta_n\subset [n-2,2n+2]\times [-an-2,an+2]$.
For any $u\in D(t_\Gamma)$ one has $u_n:=\theta_n u\in H^1(\RR_+^2\cup \RR^2_-)$ for large $n$
and $u_n$ conincides with $u$ in a neighborhood of $\supp f_n$, and then
\begin{equation}\label{domain1}
\begin{aligned}
t_{\Gamma}(u,f_n)=t_\Gamma(u_n,f_n)&=\int_{\mathbb{R}^2\setminus\Gamma }\langle\nabla u_n,\nabla f_n\rangle\dd x\dd y\\
&\qquad
-\int_{0}^{\infty}\big(\overline{u^+_n}- \overline{u^-_n}\big)(x,0)\,\big(f_{n}^{+}-f_n^-\big)(x,0)\dd x\\
&=\int_{\RR_+^2}\langle\nabla u^+_n,\nabla f_n^+\rangle\dd x\dd y+\int_{\RR_-^2}\langle\nabla u^-_n,\nabla f_n^-\rangle\dd x\dd y\\
&\qquad - \int_\RR\big(\overline{u^+_n}- \overline{u^-_n}\big)(x,0)\,\big(f_{n}^{+}-f_n^-\big)(x,0)\dd x,
\end{aligned}
\end{equation}
where $u^\pm_n$ are the restrictions of $u_n$ on the half-planes $\{\pm y>0\}$.

Denote by $\partial_\pm$ the outer normal derivative for $\RR^2_\pm$
and remark that for arbitrary functions $g$ on $\RR^2_\pm$ one has $(\partial_+ g)(x,0)=-\partial_y g(x,0)$
and $(\partial_- g)(x,0)=\partial_y g(x,0)$.
Using the integration by parts we obtain
\begin{equation}
	\label{domdom}
\begin{aligned}
	\int_{\RR_+^2}\langle\nabla u^+_n,\nabla f_n^+\rangle\dd x\dd y&= 	\int_{\RR_+^2}\overline{u^+_n} (-\Delta f_n^+)\dd x\dd y+\int_{\partial\RR^2_+} \overline{ u^+_n} \partial_+ f_n^+\dd \sigma\\
	&=\int_{\RR_+^2}\overline{u^+_n} (-\Delta f_n^+)\dd x\dd y-\int_\RR \overline{u^+_n}(x,0) \, \partial_y f_n^+(x,0)\dd x,\\
	\int_{\RR_-^2}\langle\nabla u^-_n,\nabla f_n^-\rangle\dd x\dd y&= 	\int_{\RR_-^2}\overline{u^-_n} (-\Delta f_n^-)\dd x\dd y+\int_{\partial\RR^2_-} \overline{u^-_n} \partial_- f_n^-\dd \sigma\\
&=\int_{\RR_-^2}\overline{u^-_n} (-\Delta f_n^-)\dd x\dd y+\int_\RR \overline{u^-_n}(x,0) \,\partial_y f_n^+(x,0)\dd x.
\end{aligned}
\end{equation}
Define $F_n$ by $F_n(x,y)=-\Delta f_n^\pm (x,y)$ for $\pm y>0$, then the substitution of \eqref{domdom}
into \eqref{domain1} gives, with the help of \eqref{eqff},
\begin{align*}
	t_{\Gamma}(u,f_n)&=\langle u_n,F_n\rangle_{L^2(\RR^2)}- \int_\RR\big(\overline{u^+_n}- \overline{u^-_n}\big)(x,0)\,\big(f_{n}^{+}-f_n^-\big)(x,0)\dd x\\
	&\quad -\int_\RR \overline{u^+_n}(x,0) \, \partial_y f_n^+(x,0)\dd x+\int_0^\infty \overline{u^-_n}(x,0) \,\partial_y f_n^+(x,0)\dd x\\
	&=\langle u_n,F_n\rangle_{L^2(\RR^2)}- \int_{0}^{\infty}\big(\overline{u^+_n}- \overline{u^-_n}\big)(x,0)\,\big(f_{n}^{+}-f_n^-\big)(x,0)\dd x\\
	&\quad-\int_\RR \big(\overline{u^+_n}-\overline{u^-_n}\big)(x,0)\,\partial_y f_n^+(x,0)\dd x\\
	&=\langle u_n,F_n\rangle_{L^2(\RR^2)}-\int_{0}^{\infty}\big(\overline{u^+_n}- \overline{u^-_n}\big)(x,0)\,\big(f_{n}^{+}-f_n^- + \partial_y f_n^+\big)(x,0)\dd x\\
	&=\langle u_n,F_n\rangle_{L^2(\RR^2)}=\langle u,F_n\rangle_{L^2(\RR^2)}.
\end{align*}
In the last step we have used that $u=u_n$ on $\supp F_n$. This shows that $f_n\in D(T_\Gamma)$ with $T_\Gamma f_n=F_n$.

\end{document}